\documentclass[11pt, reqno]{amsart}

\title{A characterization of Sophie Germain primes - II}

\usepackage[T1]{fontenc}
\usepackage{amsmath}
\usepackage{amssymb}
\usepackage{amsthm}
\usepackage[left=3.5cm, right=3.5cm, paperheight=11.8in]{geometry}
\usepackage{hyperref}
\usepackage{fancyhdr}
\usepackage{enumitem}
\usepackage{comment}
\usepackage{nicefrac}
\usepackage{mathrsfs}
\usepackage{bm}
\usepackage{graphicx}
\usepackage[utf8]{inputenc}
\usepackage{cancel}
\usepackage{mathtools}

\AtBeginDocument{%
   \def\MR#1{}
}

\newtheorem{thm}{Theorem}[section]

\theoremstyle{definition} 
\newtheorem{defi}[thm]{Definition}
\let\olddefi\defi
\renewcommand{\defi}{\olddefi\normalfont}
\newtheorem{example}[thm]{Example}
\let\oldexample\example
\renewcommand{\example}{\oldexample\normalfont}
\newtheorem{rmk}[thm]{Remark}
\let\oldrmk\rmk
\renewcommand{\rmk}{\oldrmk\normalfont}

\newtheorem{claim}{\textsc{Claim}}

\author[P.~Leonetti]{Paolo Leonetti}
\address{
Universit\'{a} degli Studi dell’Insubria\\ via Monte Generoso 71 \\ Varese 21100\\ Italy}
\email{leonetti.paolo@gmail.com}
\urladdr{\url{https://sites.google.com/site/leonettipaolo}}


\hypersetup{
    pdftitle={},
    pdfauthor={},
    pdfmenubar=false,
    pdffitwindow=true,
    pdfstartview=FitH,
    colorlinks=true,
    linkcolor=blue,
    citecolor=green,
    urlcolor=cyan
}

\providecommand{\MR}[1]{}

\providecommand{\MR}{\relax\ifhmode\unskip\space\fi MR }

\begin{document}

\begin{abstract} 
\noindent{} 
Let $n\ge 4$ be an even integer. We show that $\{1^{\sigma(1)},2^{\sigma(2)},\ldots,n^{\sigma(n)}\}$ is a complete residue system modulo $n$ for some permutation $\sigma$ of $\{1,2,\ldots,n\}$ if and only if $p:=n/2$ is prime and $p-1$ is squarefree. This complements the results in 
[Int. J. Number Theory \textbf{14} (2018), 653--660], where an analogue characterization has been proved for odd integers. 
\end{abstract}

\maketitle
\thispagestyle{empty}

\section{Introduction and Main result}

A \emph{Sophie Germain prime} is a prime $p$ such that $2p+1$ is also prime. The prime $2p+1$ is then called a \emph{safe prime}. They are useful mainly because safe primes give multiplicative groups with a very simple and strong structure: if $q=2p+1$ is prime, then $(\mathbb{Z}/q\mathbb{Z})^\times$ has order $2p$, hence contains a large subgroup of prime order $p$. This is valuable in discrete-logarithm cryptography, especially Diffie–Hellman-type protocols. More generally, Sophie Germain and safe primes appear in public-key cryptography, pseudorandom number generation, and primality testing, see e.g. \cite{MR2123939, MR1192408}; they also arose historically in Sophie Germain’s work on special cases of Fermat’s Last Theorem \cite{MR616635}. It is remarkable that it is currently an open question whether there exist infinitely many Sophie Germain primes. 

Hereafter, following \cite{MR3786640}, we say that an integer $n\ge 2$ is \emph{nice} if $$\{1^{\sigma(1)},2^{\sigma(2)},\ldots,n^{\sigma(n)}\}$$ is a complete residue system modulo $n$ for some permutation $\sigma$ of $\{1,2,\ldots,n\}$. 
It turns out that Sophie Germain primes are strictly related to the notion of nice numbers:

\begin{thm}\label{thm:oddold} 
\cite[Theorem 1.1]{MR3786640} 
    Let $n\ge 5$ be an odd integer. Then $n$ is nice if and only if $n$ is a safe prime \textup{(}that is, if and only if $(n-1)/2$ is a Sophie Germain prime\textup{)}.
\end{thm}

In the case of even integers $n$, the following partial result is known: 
\begin{thm}\label{thm:evenold} 
\cite[Theorem 1.2]{MR3786640} 
Let $n\ge 4$ be an even integer. If $n$ is nice, then $n=2p$ for some prime $p$ such that $p-1$ is squarefree. Conversely, if $n=2p$ for some safe prime $p\ge 7$, then $n$ is nice.
\end{thm}

Accordingly, it has been conjectured in \cite[Conjecture 1.3]{MR3786640} that an integer $n\ge 11$ is nice if and only if $n$ or $n/2$ is a safe prime. However, it has been recently shown in \cite[Theorem 5.5]{MR4932410} that the latter conjecture fails (namely, $n=62$ is nice but neither $n$ nor $n/2$ is a safe prime). 

In this note, we provide the final characterization for even integers $n$, which improves Theorem \ref{thm:evenold} (recall by \cite[Lemma 2.1]{MR3786640} that every $n \in \{2,3,\ldots7\}$ is nice):
\begin{thm}\label{thm:mainneweven} 
Let $n\ge 6$ be an even integer. Then $n$ is nice if and only if $p:=n/2$ is prime and $p-1$ is squarefree. 
\end{thm}
The proof of Theorem \ref{thm:mainneweven} follows in the next Section.


\section{Proof of Theorem \ref{thm:mainneweven}}\label{sec:proof}

\textsc{Only If part}. It follows by Theorem \ref{thm:evenold}. 

\medskip

\textsc{If part}. Suppose that $n\ge 6$ an even integer such that $p:=n/2$ is prime and $s:=p-1$ is squarefree. Let us recall that $n=6$ is nice by \cite[Lemma 2.1]{MR3786640}. Hence, let us suppose hereafter that $p\ge 5$. We shall prove that $2p$ is nice.

We start with an auxiliary result on finite abelian groups.

\begin{claim}\label{claim:twocopies}
Let $G$ be a finite abelian group. Then there exist permutations $\alpha$ and $\beta$ of $G$ such that the multiset
\begin{equation}\label{eq:multiset}
\{\alpha(x)x^{-1}:x\in G\}\cup\{\beta(x)x^{-1}:x\in G\}
\end{equation}
contains every element of $G$ exactly twice.
\end{claim}
\begin{proof}
We proceed by induction on $\# G$.

If $\# G$ is odd, define 
$\alpha(x)=\beta(x)=x^2$  
for each $x\in G$. Since $\# G$ is odd, the map $x\mapsto x^2$ is a permutation of $G$. Moreover, 
$\alpha(x)x^{-1}=\beta(x)x^{-1}=x$,  
for each $x$. Hence every element of $G$ occurs exactly twice. 

Let us suppose now that $\# G$ is even. Since $G$ is finite abelian of even order, it has a subgroup $H$ of index $2$. Choose $t\in G\setminus H$. Then 
$G=H\sqcup tH$ 
and $t^2\in H$. By the induction hypothesis, there exist permutations $\alpha_0$ and $\beta_0$ of $H$ such that
$$
\{\alpha_0(h)h^{-1}:h\in H\}\cup\{\beta_0(h)h^{-1}:h\in H\}
$$
contains every element of $H$ exactly twice.

Define maps $\alpha,\beta:G\to G$ by
$$
\alpha(h)=\alpha_0(h),\text{ }\text{ }\alpha(th)=t\beta_0(h)
$$
and
$$
\beta(h)=t\alpha_0(h),\text{ }\text{ }\beta(th)=t^2\beta_0(h)
$$
for each $h\in H$. 
Then $\alpha$ and $\beta$ are permutations of $G$. 
Indeed, $\alpha$ 
maps $H$ bijectively onto $H$ and $tH$ bijectively onto $tH$. Similarly, $\beta$ 
maps $H$ bijectively onto $tH$ and $tH$ bijectively onto $H$.

For $h\in H$, we have
$$
\alpha(h)h^{-1}=\alpha_0(h)h^{-1}.
$$
Moreover, using that $G$ is abelian,
$$
\alpha(th)(th)^{-1}=
t\beta_0(h)h^{-1}t^{-1}=
\beta_0(h)h^{-1}.
$$
Therefore the quotients coming from $\alpha$ contain every element of $H$ exactly twice.

At the same time,
$$
\beta(h)h^{-1}=t\alpha_0(h)h^{-1}
$$
and
$$
\beta(th)(th)^{-1}=
t^2\beta_0(h)h^{-1}t^{-1}=
t\beta_0(h)h^{-1}.
$$
Therefore the quotients coming from $\beta$ contain every element of $tH$ exactly twice. Hence the whole multiset 
defined in \eqref{eq:multiset} 
contains every element of $G$ exactly twice. 
\end{proof}

We shall use Claim \ref{claim:twocopies} to construct suitable residues for the exponents.

\begin{claim}\label{claim:squarefreeresidues}
There exist two maps 
$\rho_1,\rho_2:\mathbb{Z}/s\mathbb{Z}\to \mathbb{Z}/s\mathbb{Z}$  
such that the multiset
$$
\{\rho_1(x):x\in \mathbb{Z}/s\mathbb{Z}\}\cup\{\rho_2(x):x\in \mathbb{Z}/s\mathbb{Z}\}
$$
contains every element of $\mathbb{Z}/s\mathbb{Z}$ exactly twice, and both maps
$$
x\mapsto x\rho_1(x)
\quad \text{ and }\quad 
x\mapsto x\rho_2(x)
$$
are permutations of $\mathbb{Z}/s\mathbb{Z}$.
\end{claim}
\begin{proof}
For each divisor $d$ of $s$, define 
$$
X_d=\{x\in \mathbb{Z}/s\mathbb{Z}:\mathrm{gcd}(x,s)=d\},
$$ 
where $\mathrm{gcd}(x,s)$ means the greatest common divisor of $s$ with any integer representative of $x$. The sets $X_d$, as $d$ ranges over the divisors of $s$, form a partition of $\mathbb{Z}/s\mathbb{Z}$.

First, consider $d=s$. Then $X_s=\{0\}$, and we set
$$
\rho_1(0)=\rho_2(0)=0.
$$

Now, suppose that $d<s$, and define $m=s/d$.  
Since $s$ is squarefree, we have $\mathrm{gcd}(d,m)=1$. Moreover, every element $x\in X_d$ can be written uniquely in the form 
$
x=du\pmod{s},
$ 
where $u$ belongs to the group of units of $G_d=(\mathbb{Z}/m\mathbb{Z})^\times$. 
Thanks to Claim \ref{claim:twocopies}, there exist permutations $\alpha_d$ and $\beta_d$ of $G_d$ such that
$$
\{\alpha_d(u)u^{-1}:u\in G_d\}\cup\{\beta_d(u)u^{-1}:u\in G_d\}
$$
contains every element of $G_d$ exactly twice.

We define a map 
$\theta_d:G_d\to \mathbb{Z}/s\mathbb{Z}$ 
as follows. For $r\in G_d$, let $\theta_d(r)$ be the unique residue modulo $s$ satisfying
$$
\theta_d(r)\equiv 0\pmod d
\quad \text{ and }\quad 
\theta_d(r)\equiv r\pmod m,
$$
which is well defined by the Chinese remainder theorem. 

We claim that $\theta_d$ is a bijection from $G_d$ onto $X_d$. Indeed, if $q$ is a prime divisor of $d$, then $\theta_d(r)\equiv 0\pmod q$. If $q$ is a prime divisor of $m$, then $\theta_d(r)\equiv r\not\equiv 0\pmod q$, since $r$ is a unit modulo $m$. Since $s=dm$ is squarefree, it follows that 
$\mathrm{gcd}(\theta_d(r),s)=d$.  
Thus $\theta_d(r)\in X_d$. 
Conversely, let $e\in X_d$. Then, by definition, 
$\mathrm{gcd}(e,s)=d$.  
Since $s=dm$ and $\mathrm{gcd}(d,m)=1$, it follows that
$e\equiv 0\pmod d$  
and that the reduction of $e$ modulo $m$ is a unit. In other words,
$e\bmod m\in G_d$.  
By construction, $\theta_d(e\bmod m)$ is the unique residue class modulo $s$ satisfying
$$
\theta_d(e\bmod m)\equiv 0\pmod d
\quad \text{ and }\quad 
\theta_d(e\bmod m)\equiv e\pmod m.
$$
But the residue class $e$ itself satisfies these two congruences. Hence, by the uniqueness part of the Chinese remainder theorem, 
$\theta_d(e\bmod m)=e$  
in $\mathbb Z/s\mathbb{Z}$. Therefore every element of $X_d$ belongs to the image of $\theta_d$, and $\theta_d$ is a bijection from $G_d$ onto $X_d$.

At this point, for $x\in X_d$, write $x=du$ with $u\in G_d$, and define
$$
\rho_1(x)=\theta_d(\alpha_d(u)u^{-1})
\quad \text{ and }\quad 
\rho_2(x)=\theta_d(\beta_d(u)u^{-1}).
$$
Then 
$\rho_1(x)\equiv \alpha_d(u)u^{-1}\pmod m$,  
hence 
$u\rho_1(x)\equiv \alpha_d(u)\pmod m$.  
Multiplying by $d$, we obtain
$$
x\rho_1(x)=du\rho_1(x)\equiv d\alpha_d(u)\pmod{s}.
$$
As $u$ runs through $G_d$, also $\alpha_d(u)$ runs through $G_d$. Therefore the values $x\rho_1(x)$, with $x\in X_d$, run through exactly $X_d$. The same argument applies to the values $x\rho_2(x)$.

Since this holds for every divisor $d$ of $s$, and since for $d=s$ we have $0\rho_1(0)=0\rho_2(0)=0$, both maps
$x\mapsto x\rho_1(x)$ and $x\mapsto x\rho_2(x)$ are permutations of $\mathbb{Z}/s\mathbb{Z}$. 

It remains to check the multiplicities of the values of $\rho_1$ and $\rho_2$. For $d=s$, the value $0$ occurs twice. For $d<s$, thanks to Claim \ref{claim:twocopies}, the values 
$\alpha_d(u)u^{-1}$ and $\beta_d(u)u^{-1}$,  
as $u$ runs through $G_d$, contain every element of $G_d$ exactly twice. Lastly, since $\theta_d$ is a bijection from $G_d$ onto $X_d$, the values
$\rho_1(x)$ and $\rho_2(x)$,  
as $x$ runs through $X_d$, contain every element of $X_d$ exactly twice. Taking into account the sets $X_d$ form a partition $\mathbb{Z}/s\mathbb{Z}$, this completes the proof.
\end{proof}

To finish the proof, let $\xi$ be a generator of the group of units of $\mathbb{Z}/p\mathbb{Z}$. For each $x\in \mathbb{Z}/s\mathbb{Z}$, let $a_x$ be the unique odd integer in $\{1,\ldots,2p\}$ and let $b_x$ be the unique even integer in $\{1,\ldots,2p\}$ such that
$$
a_x\equiv b_x \equiv \xi^x\pmod p.
$$
Of course, $a_x$ and $b_x$ are well defined. Moreover,
$$
\{a_x:x\in \mathbb{Z}/s\mathbb{Z}\}\cup\{b_x:x\in \mathbb{Z}/s\mathbb{Z}\}=
\{1,\ldots,2p\}\setminus\{p,2p\}.
$$

Let $\rho_1,\rho_2:\mathbb{Z}/s\mathbb{Z}\to \mathbb{Z}/s\mathbb{Z}$ be the maps given by Claim \ref{claim:squarefreeresidues}. Since 
$2p=2s+2$, the integers $3,4,\ldots,2p$ contain every residue class modulo $s$ exactly twice. On the other hand, by Claim \ref{claim:squarefreeresidues}, the multiset
$$
\{\rho_1(x):x\in \mathbb{Z}/s\mathbb{Z}\}\cup\{\rho_2(x):x\in \mathbb{Z}/s\mathbb{Z}\}
$$
also contains every residue class modulo $s$ exactly twice. Hence we can choose pairwise distinct integers 
$e_x,f_x\in\{3,4,\ldots,2p\}$ for each $x\in \mathbb{Z}/s\mathbb{Z}$ 
such that
$$
\forall x\in \mathbb{Z}/s\mathbb{Z}, \qquad 
e_x\equiv \rho_1(x)\pmod s
\quad \text{ and }\quad 
f_x\equiv \rho_2(x)\pmod s.
$$

\medskip

At this point, define a permutation $\sigma$ of $\{1,\ldots,2p\}$ by $\sigma(p)=1$, $\sigma(2p)=2$, and 
$$
\forall x \in \mathbb{Z}/s\mathbb{Z}, \qquad 
\sigma(a_x)=e_x \quad \text{ and }\quad \sigma(b_x)=f_x.
$$ 
This is a permutation because the numbers $p,2p$, together with the $a_x$'s and $b_x$'s, are all the elements of $\{1,\ldots,2p\}$, while the assigned values are exactly $1,2,\ldots,2p$.

It is clear that $p^{\sigma(p)}\equiv p\pmod {2p}$ and $(2p)^{\sigma(2p)}\equiv 0\pmod {2p}$. 
Now fix $x\in \mathbb{Z}/s\mathbb{Z}$. Since $e_x\equiv \rho_1(x)\pmod s$ and $s=p-1$, Fermat's little theorem gives
$$
a_x^{\sigma(a_x)}= a_x^{e_x}
\equiv (\xi^x)^{e_x}
\equiv \xi^{x\rho_1(x)} \pmod p.
$$
Moreover, $a_x^{e_x}$ is odd. Since the map 
$x\mapsto x\rho_1(x)$ is a permutation of $\mathbb{Z}/s\mathbb{Z}$, the residues
$$
\{a_x^{\sigma(a_x)}: x\in \mathbb{Z}/s\mathbb{Z}\} 
$$ 
are exactly the odd residues modulo $2p$ which are not divisible by $p$. 
With an analogous reasoning, we have $b_x^{\sigma(b_x)}\equiv \xi^{x\rho_2(x)}\pmod p$, hence 
the residues
$$
\{b_x^{\sigma(b_x)}: x\in \mathbb{Z}/s\mathbb{Z}\}
$$
are exactly the even residues modulo $2p$ which are not divisible by $p$.
 
It follows by construction that
$$
\{1^{\sigma(1)},2^{\sigma(2)},\ldots,(2p)^{\sigma(2p)}\}
$$
is a complete residue system modulo $2p$. Therefore $n$ is nice. This concludes the proof of Theorem \ref{thm:mainneweven}. 


\subsection{Acknoweledgments} The author is grateful to Krzysztof Kowitz (University of Gdańsk) for useful discussions on the topic.

\bibliographystyle{amsplain}
\bibliography{refs}

\end{document}